\documentclass[10pt,a4paper]{article}

\usepackage{color}

\usepackage{url}
\usepackage{verbatim}
\usepackage{latexsym}
\usepackage{amssymb,amstext,amsmath,mathtools,amsthm,thmtools}
\usepackage{hyperref}
\usepackage[capitalise]{cleveref}
\usepackage{xfrac}
\usepackage{enumitem}
\usepackage{epsf}
\usepackage{epsfig}
\usepackage{a4wide}
\usepackage{verbatim}
\usepackage{proof}
\usepackage{latexsym}
\newtheorem{theorem}{Theorem}[section]
\newtheorem{corollary}[theorem]{Corollary}
\newtheorem{lemma}[theorem]{Lemma}

\theoremstyle{definition}

\theoremstyle{remark}

\usepackage{float}
\floatstyle{boxed}
\restylefloat{figure}

\def\oge{\leavevmode\raise
.3ex\hbox{\(\scriptscriptstyle\langle\!\langle\,\)}}
\def\feg{\leavevmode\raise
.3ex\hbox{\(\scriptscriptstyle\,\rangle\!\rangle\)}}

\newcommand{\nats}{\mathbb{N}}

\newcommand{\ints}{\mathbb{Z}}

\newcommand{\Der}{\mathsf{Der}}
\newcommand{\ev}{\mathsf{ev}}
\newcommand{\nf}{\mathsf{nf}}

\newcommand{\CC}{\mathcal{C}}

\newcommand{\Alg}{\mathsf{Alg}}

\DeclarePairedDelimiter\truncImpl{\lVert}{\rVert}
\DeclareDocumentCommand\trunc{sm}{\truncImpl{#2}\IfBooleanT{#1}{_{\mathsf{S}}}}

\DeclareMathOperator\id{id}

\newcommand\RR{\mathsf{R}}

\DeclareDocumentCommand\Alg{so}{\IfNoValueTF{#2}{\mathsf{Alg}}{#2\textsf{-}\mathsf{Alg}}\IfBooleanT{#1}{_{fp}}}

\usepackage{tikz,tikz-cd}
\tikzcdset{cells={font=\everymath\expandafter{\the\everymath\displaystyle}}}
\tikzset{weq/.style = {"\sim"'{sloped,font=\tiny,#1}}}

\usepackage{todonotes}

\begin{document}

\title{A Note About Models of Synthetic Algebraic Geometry}

\author{Thierry Coquand, Jonas H\"ofer and Christian Sattler}
\date{}
\maketitle


\section*{Introduction}

We show how to build models of the system presented in \cite{draft}
for special base commutative ring.
At the same time, we report a mistake,
and a way to fix this mistake in important cases, in the construction of a model of the axiom system  presented in~\cite{draft}.
This mistake was found by the second author, and a solution for the general case will be presented in forthcoming work.

Interestingly, the problem involves the notion of {\em set quotient} in the underlying metatheory, and how it compares with quotient in models
of homotopy type theory. This is an issue important in the study of
{\em constructive} models of type theory with univalence \cite{shulman21}. And indeed, the solution we present in a special case has the crucial feature that it does {\em not} require set quotients in the metatheory (and this will also holds for the general case).

In the paper \cite{draft}, the justification of the axiom system involved first a construction of a {\em presheaf model} satisfying a system of
4 axioms. This was checked only in the {\em $1$-topos} case, arguing that, since the objects involved are all strict/constant presheaves, this justification would
lift to the case of the {\em cubical} presheaf model. As we explain in the first section, this is however {\em not} correct, essentially because the set
quotient required for building finitely presented $k$-algebras does not behave as it should as homotopy (h)set quotient. This lifting from sets to cubical sets
{\em does} work in some cases, e.g.
if the base ring is a discrete field \cite{mines} or $\ints$. Even in these cases however, one has further to introduce a new idea, which is to
{\em relativise} the presheaf model using the cobar modality introduced in \cite{CRS21}.

\section{Description of the problem}

Let $k$ be a fixed ring and $\CC$ the opposite of the category
of finitely presented $k$-algebras. 
We write $L,M,N,\dots$ for such objects and $f,g,h,\dots$ for the $k$-algebra morphisms.
A cubical presheaf $F$ on this category $\CC$
can be seen as a family of cubical sets $F(L)$ with restriction maps $F(L)\rightarrow F(M),~u\mapsto fu$ for $f:L\rightarrow M$ satistying the
{\em strict} equalities $\id~u= u$ and $(g\circ f)~u = g(f~u)$. A natural transformation $u:F\rightarrow G$ is given by a family of
cubical set maps $u_L:F(L)\rightarrow G(L)$ satisfying {\em strict} commutation $f(u_La) = u_M(f a)$ for $f:L\rightarrow M$ and $a:F(L)$\footnote{This is an internal statement in the presheaf of cubical sets, involving a quantification on the cubical set $F(L)$.}.
If $E$ is a set, we write $\Delta(E)$ the associated constant cubical set.
We can see $\Delta$ as a functor from sets to cubical sets, and, if $F$ is a {\em set valued} presheaf, consider $\Delta\circ F$ which is now
a cubical presheaf.
The model described in \cite{draft} is the internal model of fibrant cubical presheaves \cite{draft,CRS21} over $\CC$.

We define $\RR$ by taking $\RR(M)$ to be $\Delta(M)$. This defines (internally) a ring, and the claim in~\cite{draft}
was that this ring satisfies suitable axioms, involving finitely presented $\RR$-algebras.

The problem is in the description of these
finitely presented $\RR$-algebras. It turns out that this problem can already be seen in the simple case $P = \RR[X]$.
We are  going to explain this case in detail.

In \cite{draft}, it was assumed implicitely that we could define $P$ simply as $P(L) = \Delta (L[X])$. This defines a presheaf,
which is a $\RR$-algebra. 
The problem is that this
$\RR$-algebra is {\em not} $\RR[X]$ in general for the homotopical semantics, i.e.
it does not satisfy the required universal property of $\RR[X]$ in the presheaf model.
Indeed, for this universal property, we should be able, in particular, given a cubical presheaf ring
$A$ with a ring morphism $u:\RR\rightarrow A$ and a global element $a:A$, to extend
it to a map $v:P\rightarrow A$. This would require
\begin{enumerate}
\item a {\em levelwise} problem, to define for each $L$ a map $v_L:\Delta (P(L))\rightarrow A(L)$
\item a {\em strict coherence} problem, to check that these maps commute {\em strictly} with restrictions, i.e. $v_M(f~p) = f v_L(p)$, as a
  strict equality, for any $p$ in $P(L)$ and $f:L\rightarrow M$.
\end{enumerate}
It turns out that there are issues in general with {\em both} requirements.

\subsection{Levelwise Problem}


One way to understand this problem is the following. The algebra $L[X]$ in {\em sets}, can be seen as the set quotient
of the free term algebra $E$, with ring operations on a variable $X$ and constants $l$ for each element $l$ of $L$, by the
equational theory of rings. We
can define the evaluation map $E\rightarrow A(L)$ given $u_L$ and $a_L$. {\em If} the ring equations hold strictly in $A(L)$,
this map will factor through the quotient map $\Delta(L[X])\rightarrow A(L)$. However, in general, these laws hold only
as path equality and there is no way to define such a factorization.

This problem can be  summarized informally by stating that we may not have the equality
\[
\Delta(L[X]) = \Delta(L)[X]
\]
in general if $L$ is a set theoretic ring, and $L[X]$ the polynomial ring in sets.

This issue does not appear however if the equality on $L$ is {\em decidable}. In this case, we can choose a unique way to write an element of
$L[X]$, e.g. as $l_0+\dots + l_nX^n$ with $l_n\neq 0$, and it is now possible to define the universal map $u_L$ without ambiguity, as 
a factorization of the evaluation map $E\rightarrow A(L)$.

\subsection{Strict Coherence Problem}


The second issue is now a strict coherence problem. Even in the decidable case,
we only can expect the equality $u_M(f~p) = f u_L(p)$ to hold as a {\em path} equality (and not as
a strict equality).

\section{A solution if $k$ is a finitely presented algebra over $\ints$ or a discrete field}

\subsection{The cobar modality}

In \cite{CRS21}, we introduce a left exact modality \cite{modalities}
$D$ (which can be defined on any presheaf model), the {\em cobar} modality. We write $\eta_F:F \rightarrow DF$
the unit of this modality. The following Lemma is a consequence of the properties proved in \cite{CRS21}.

\begin{lemma}
  \begin{enumerate}
   \item Given two presheaves $F$ and $G$, and a family of maps $v_L:F(L)\rightarrow G(L)$ with {\em path} commutation w.r.t. restriction, {\em if} each $G(L)$
is a (h)set, then there exists a {\em strict} natural transformation $w:F\rightarrow DG$ such that $\eta_F\circ v_L$ is path equal to $w_L$.
    \item If $F$ is a set valued presheaf, then the associated cubical presheaf $\Delta \circ F$ is cobar modal.
  \end{enumerate}
\end{lemma}

This suggests that a solution (in the case where each $L$ has a decidable equality) should be to work with the relativised model of
{\em cobar modal} presheaves.

\subsection{A model}

We build a model of the axioms 1-4 discussed in \cite{draft} in the case where $k$
is a finitely presented algebra over $\ints$ or a discrete field, using the cobar relativisation of the presheaf models over $\CC$
as a underlying model of type theory.
We use the same interpretation of $\RR$ as in \cite{draft}, namely defined by $\RR(L) = \Delta(L)$, which is already cobar modal.

In order to do this, we use a general Lemma\footnote{The levelwise problem can conceptually be formulated at the level of the left exact
completion of the category of sets, where objects are pairs $(X,R)$, and $R$ being a ``proof-relevant'' equivalence relation on the set $X$.
In general, if $x_0\sim x_1$ means that the set $R(x_0,x_1)$ is inhabited, then the pair $(X,R)$ is not equivalent to the pair $(X/\sim,=)$.
This is closely connected to the discussion in the introduction of \cite{shulman21}.}
about models of equational theories in cubical sets. Let $T$
be a (first-order) equational theory. Let $E$ be the set of closed terms, and, for $t,u$ in $E$, let $\Der(t,u)$ be the set
of equational proofs of equality of $t$ and $u$ in the theory $T$. The initial term model, in {\em sets}, is the quotient of $E$
by the equivalence relation $t\sim u$, which states that there exists an element in $\Der(t,u)$.


\begin{lemma}\label{norm}
  If there is a normal form function\footnote{This means that $\nf$ is an idempotent function on the set $E$ such that
  $t\sim u$ if, and only if, $\nf(t) = \nf(u)$.} $\nf:E\rightarrow E$ and a section\footnote{Given this function, we get a choice
  function which, from $t\sim u$, extracts an element of $\Der(t,u).$} $\Pi_{t:E}\Der(t,\nf(t))$ then the constant cubical
  set $\Delta(E/\sim)$ is the initial (h)set term model of $T$ in the cubical model.
\end{lemma}

\begin{proof}
Notice that, in this case, we don't need quotient
in the set theory to form $E/\sim$ since it can be built as the subset $N$ of $E$ of terms in normal forms, with a function
$\nf:E\rightarrow N$. If $A$ is a (h)set model of the theory $T$ there is an evaluation function $\ev:\Delta(E)\rightarrow T$, which
satisfies $\ev(t) =_A \ev(u)$ as a {\em path equality} given an element of $\Der(t,u)$.
This gives the initial map $i:\Delta(N)\rightarrow A$ which satisfies $i\circ \Delta(\nf) = ev$ as a path equality given the section
$\Der(t,\nf(t))$.
\end{proof}

To have a section $\Pi_{t:E}\Der(t,\nf(t))$ can be seen as a constructive requirement: we not only should have $t\sim \nf(t)$ but also
an explicit derivation $\Der(t,\nf(t))$ in order to produce a path equality between $\ev(t)$ and $\ev(\nf(t))$ in $A$.


\begin{corollary}\label{quot}
  If $L$ is a finitely presented algebra over a discrete field or $\ints$ then
  \[
  \Delta (L[X_1,\dots,X_n]/(q_1,\dots,q_m))
  \]
  is, in the model of cubical sets, the (h)set quotient
  of $(\Delta L)[X_1,\dots,X_n]$ by $(q_1,\dots,q_m).$
\end{corollary}

\begin{proof}
  We can use Gr\"obner bases \cite{yengui2006}, with a constructive proof of correctness, to get a suitable normal form function
  and then apply the previous Lemma.
\end{proof}

\newcommand{\FP}{\mathsf{FP}}
\newcommand{\UU}{\mathcal{U}}

Using this, we can follow the argument in \cite{draft}, Section 8.1.2. We now work in the presheaf
model over $\CC$, relativised by the cobar modality.
We use the presheaf of ``finite presentations'' which is internally the type
\(\Sigma_{n:\mathbb{N}}\Sigma_{m:\mathbb{N}}\RR[X_{1},\ldots,X_{n}]^{m}\) and externally the type
\[
\FP(L) = \Sigma_{n:\nats}\Sigma_{m:\nats}\Delta (L[X_1,\dots,X_n]^m).
\]
(note that this type is cobar modal).
We can now define the 
dependent presheaf which, to any presentation
$\xi = (n,m,q_1,\dots,q_m)$, associates the $\RR$-algebra $A(\xi) = \RR[X_1,\dots,X_n]/(q_1,\dots,q_m)$.
Externally, we define it as the presheaf
\[
A(L,\xi) = \Delta (L[X_1,\dots,X_n]/(q_1,\dots,q_m)) = (\Delta L)[X_1,\dots,X_n]/(q_1,\dots,q_m)
\]
The last equality follows from Corollary \ref{quot}, and {\em if} we use the model of cobar modal presheaves, this is the correct
external description of $A$. We can then justify axioms 2-4 as in \cite{draft}.

\section{Another example}

The technique presented in this paper applies in the case of {\em finitely presented distributive lattices} instead of
finitely presented $k$-algebra, since we can apply Lemma \ref{norm} in this case. This should provide a model
of the theory presented in \cite{gratzer2024} in a constructive meta theory. Note that the cobar modality is also used
in a crucial way in the work \cite{weaver20}.


\begin{thebibliography}{1}

\bibitem{draft}
Felix {Cherubini}, Thierry {Coquand}, and Matthias {Hutzler}.
\newblock A foundation for synthetic algebraic geometry.
\newblock {\em Mathematical Structures in Computer Science}, 34(9):1008--1053,
  2024.

\bibitem{CRS21}
Thierry Coquand, Fabian Ruch, and Christian Sattler.
\newblock Constructive sheaf models of type theory.
\newblock {\em Math. Struct. Comput. Sci.}, 31(9):979--1002, 2021.

\bibitem{gratzer2024}
Daniel Gratzer, Jonathan Weinberger, and Ulrik Buchholtz.
\newblock Directed univalence in simplicial homotopy type theory, 2024.

\bibitem{mines}
Ray Mines, Fred Richman, and Wim Ruitenburg.
\newblock {\em A Course in Constructive Algebra}.
\newblock Universitext. Springer New York, 1988.

\bibitem{modalities}
Egbert Rijke, Michael Shulman, and Bas Spitters.
\newblock {Modalities in homotopy type theory}.
\newblock {\em {Logical Methods in Computer Science}}, {Volume 16, Issue 1},
  January 2020.

\bibitem{shulman21}
Michael Shulman.
\newblock The derivator of setoids.
\newblock {\em Cahiers de topologie et géométrie différentielle
  catégoriques}, LXIV:29--96, 2021.

\bibitem{weaver20}
Matthew~Z. Weaver and Daniel~R. Licata.
\newblock A constructive model of directed univalence in bicubical sets.
\newblock In {\em Proceedings of the 35th Annual ACM/IEEE Symposium on Logic in
  Computer Science}, LICS '20, page 915–928, New York, NY, USA, 2020.
  Association for Computing Machinery.

\bibitem{yengui2006}
Ihsen Yengui.
\newblock {A dynamical solution of Kronecker's problem}.
\newblock In Thierry Coquand, Henri Lombardi, and Marie-Fran\c{c}oise Roy,
  editors, {\em Mathematics, Algorithms, Proofs}, volume 5021 of {\em Dagstuhl
  Seminar Proceedings (DagSemProc)}, pages 1--9, Dagstuhl, Germany, 2006.
  Schloss Dagstuhl -- Leibniz-Zentrum f{\"u}r Informatik.

\end{thebibliography}
\end{document}